\documentclass[12pt,a4paper,oneside,reqno]{amsart}

\usepackage{amsmath}
\usepackage{mathtools}

\allowdisplaybreaks

\usepackage{amsfonts}
\newcommand{\R}{{\mathbb R}}
\newcommand{\Lc}{{\mathfrak L}}

\DeclareMathOperator{\mes}{mes}
\DeclareMathOperator{\supp}{supp}
\DeclareMathOperator{\diam}{diam}

\newcommand*{\charac}[1]{\chi_{_{\scriptstyle#1}}}

\newcommand{\ds}{\,ds}
\newcommand{\dt}{\,dt}

\newcommand{\parameter}{\,\cdot\,}

\newcommand*{\Curly}[1]{\mathinner{\mathopen\{#1\mathclose\}}}
\newcommand*{\Norm}[1]{\mathinner{\mathopen\Vert#1\mathclose\Vert}}
\newcommand*{\Abs}[1]{\mathinner{\mathopen\vert#1\mathclose\vert}}

\IfFileExists{cite.sty}
  {
    \usepackage{cite}

  }{}

\newcommand{\iref}{\eqref}

\numberwithin{equation}{section}

\usepackage{amsthm}
\newtheorem{theorem}{Theorem}[section]
\newtheorem{lemma}{Lemma}[section]
\newtheorem{proposition}{Proposition}[section]
\newtheorem{corollary}{Corollary}[section]
\theoremstyle{definition}
\newtheorem{definition}{Definition}[section]
\newtheorem{remark}{Remark}[section]

\newtheorem{example}{Example}[section]

\providecommand{\bysame}{\leavevmode\hbox to3em{\hrulefill}\thinspace}
\newcommand{\biblanguage}[1]{\textrm{\upshape(in #1)}}

\makeatletter
\renewcommand{\subjclass}[2][1991]{\def\@subjclass{#2}\def\subjclassname
{\textup{#1} Mathematics Subject Classification}}
\makeatother

\begin{document}

\title{Degeneracy Results for Fully Nonlinear Integral Operators}

\dedicatory{In memory of W.\,A.\,J. Luxemburg and P.\,P.\ Zabre\u{\i}ko}

\author[M.~V\"{a}th]{Martin V\"{a}th}
\address{Martin V\"{a}th\\
Google Switzerland\\
Brandschenkestrasse 110\\
CH-8002 Zürich\\
Switzerland}
\email{martin@mvath.de}



\begin{abstract}
It is shown that integral operators of the fully nonlinear type
$K(x)(t)=\int_\Omega k(t,s,x(t),x(s))\ds$ exhibit similar degeneracy phenomena
in a large class of spaces as superposition operators $F(x)(t)=f(t,x(t))$.
In particular, $K$ is Fr\'{e}chet differentiable in $L_p$ only if it is
affine with respect to the ``$x(t)$'' argument. Similar degeneracy results
hold if $K$ satisfies a local Lipschitz or compactness condition.
Also vector functions, infinite measure spaces, and a much richer class
of function spaces than only $L_p$ are considered.
As a side result, degeneracy assertions for superposition operators
are obtained in this more general setting, complementing the known results
for scalar functions. As a particular example, it is shown that the operators
arising in continuous limits of coupled Kuramoto oscillators fail everywhere
to be Fr\'{e}chet differentiability or locally compact.
\end{abstract}

\maketitle

\section{Introduction}\label{s:intro}

In several recent publications, there occur integral operators of the type
\begin{equation}\label{e:fully}
K(x)(t)=\int_\Omega k(t,s,x(t),x(s))\ds\quad(t\in\Omega)
\end{equation}
($\Omega$ being some positive measure space), which are sometimes called
\emph{nonstandard integral operators}, see e.g.~\cite{GuanVolterra}.
These operators occur in rather different contexts,
for instance as a continuous limit for
coupled Kuramoto oscillators~\cite{AbramsStrogatz,Girnyk,Medvedev1,Medvedev2},
or in the modeling of American stock options in financial
mathematics, see e.g.~\cite{PeskirEconomy,KimEconomy}.

Note that the special ``bilinear'' case of~\eqref{e:fully} occurs in
so-called quadratic integral equations
\[x(t)=x(t)\int_\Omega\widetilde k(t,s)x(s)\ds+f(t)\quad(t\in\Omega)\text,\]
but the general form~\eqref{e:fully} allows e.g.\ even for arbitrary powers
in the arguments, so that the operators are actually ``fully nonlinear'',
which the author considers a better name for these operators than
``nonstandard'', because there is really nothing ``nonstandard''
about these operators (in particular, there is no special relation
to nonstandard analysis).

It is tempting to consider these fully nonlinear integral
operators~\eqref{e:fully} as ``straightforward'' generalizations of
Urysohn operators
\begin{equation}\label{e:Kone}
K_1(x)(t)=\int_\Omega k_1(t,s,x(s))\ds\quad(t\in\Omega)
\end{equation}
and to attempt similar approaches to solve equations with $K$ as for equations
with $K_1$. In particular, one might conjecture that under reasonable
regularity conditions for $k$ the operator $K$ is differentiable and compact
or at least satisfies a local (though usually not global) Lipschitz condition.
However, already for quadratic integral equations, it is known
that things are not that easy, and we will now show in this paper why such
straightforward attempts are usually doomed to fail.

In fact, one of the aims of this paper is to show that the operator which
arises as a continuous limit for coupled Kuramoto oscillators, though it
has a very regular kernel function, is neither Fr\'{e}chet differentiable
nor (locally) compact in any of the standard spaces $L_p(\Omega)$ with
$1<p<\infty$.

However, our results are much more general:
We will show that actually $K$ does not so much inherit
the ``nice'' properties of a Urysohn operator $K_1$
(like differentiability, local Lipschitz dependency, or compactness)
in many spaces under decent regularity conditions on $k_1$,
but actually that $K$ inherits much more the inconvenient properties
of superposition operators
\begin{equation}\label{e:Kzero}
K_0(x)(t)=k_0(t,x(t))\quad(t\in\Omega)\text.
\end{equation}
Recall that such superposition operators $K_0$ exhibit very unpleasant
degeneracy phenomena in many spaces. For instance, if $\Omega$ has not atoms
then $K_0$ fails to be (Fr\'{e}chet) differentiable in $L_p(\Omega)$
$(1\le p<\infty)$ unless there is a degeneracy in the sense that $K_0$
is an affine map. Similarly, a \emph{local} Lipschitz condition for $K_0$
readily implies a \emph{global} Lipschitz condition, even pointwise.
Moreover, even a local Lipschitz-type condition for
measures of noncompactness (``Darbo condition'')
implies that $K_0$ satisfies a Lipschitz condition with the same constant. In
particular, if $K_0$ is (locally) compact then $K_0$ is actually constant.

We will show analogous results for~\eqref{e:fully}, even for the generalized
fully nonlinear integral operator $F=K_0+K_1+K$, that is, we show additionally
that these degeneration phenomena cannot be ``canceled'' by just adding
superposition or Urysohn operators.

In contrast to the ``easy'' case of superposition operators,
we have to impose some additional regularity assumptions for our results,
and moreover, we have to relax the degeneration result concerning the
Darbo condition by obtaining only a global growth condition
(instead of a global Lipschitz condition).
However, we will show that all these assumptions and modificiation
are in particular sufficient to treat the operators arising from
Kuramoto oscillators mentioned above.

\section{Spaces and Notations}

Throughout, let $(E_1,\Abs\parameter)$, $(E_2,\Abs\parameter)$
be Banach spaces (not the trivial space $\Curly0$),
and $\Omega$ be a (positive) measure space (nontrivial: $\mes\Omega>0$).
We call a function $x\colon\Omega\to E_i$ measurable if it is (strongly)
Bochner measurable with respect to the Lebesgue extension of the measure space.
Let $(X,\Norm\parameter_X)$ and $(Y,\Norm\parameter_Y)$ be normed spaces,
consisting of (equivalence classes of) measurable functions
$x\colon\Omega\to E_1$ or $y\colon\Omega\to E_2$, respectively,
where we identify functions coinciding almost everywhere, as customary.

In particular, $\supp x\coloneqq\Curly{t\in\Omega:x(t)\ne0}$ is defined
up to a null set.

For measurable $D\subseteq\Omega$, we define the characteristic function
\[\charac D(t)\coloneqq\begin{cases}
1&\text{if $t\in D$,}\\
0&\text{if $t\notin D$,}
\end{cases}\]
and for a measurable function $x\colon\Omega\to E$ ($E$ some Banach space),
we define $P_Dx(t)=\charac D(t)x(t)$.

We do \emph{not} require that $P_D\colon X\to X$ (or $P_D\colon Y\to Y$).
As a substitute, we denote for $x\in X$ by $P_{D,x}$ the set of all
measurable functions $x_D\colon\Omega\to E_1$ satisfying $P_Dx_D=P_Dx$
(that is, $x_D|_D=x|_D$ almost everywhere). In particular, $P_Dx\in P_{D,x}$.

In general, $x\in X$ does not imply $P_Dx\in X$, but the
set $P_{D,x}\cap X$ is then trivially nonempty (because it contains $x$).

In~\cite{Appell} it was shown that the differentiability and Lipschitz
properties of superposition operators $F\colon X\to Y$ degenerate
if $(X,Y)$ is a so-called \emph{$V$-pair}.
At a first glance, the following technical
definition looks rather different than a $V$-pair as defined
in~\cite[Section~2.6]{Appell},
but we will see later that it is indeed a proper generalization.

\begin{definition}\label{d:Vpair}
Let $X,Y$ be spaces of measurable functions as introduced above.
We call $(X,Y)$ a \emph{weak $V$-pair}, if the following holds.

For every $x\in X$ and $y\in Y$ and every measurable subset $T\subseteq\Omega$
with $P_Ty\ne0$ there are sequences
$D_n\subseteq T$ (measurable) and $x_n\in P_{D_n,x}\cap X$ such that
$P_{D_n}y\ne0$ and
\begin{equation}\label{e:Vlimits}
\lim_{n\to\infty}\,\Norm{x_n}_X=0\quad\text{and}\quad
\liminf_{n\to\infty}\,\sup_{y_n\in P_{D_n,y}\cap Y}\,
\frac{\Norm{x_n}_X}{\Norm{y_n}_Y}<\infty\text.
\end{equation}
\end{definition}

To illustrate this abstract definition and to explain the relation with
$V$-pairs from~\cite[Section~2.6]{Appell}, we consider a special class
of spaces.

Recall that a normed space $(X,\Norm\parameter)$ of (classes of)
measurable functions $x\colon\Omega\to E_1$ is called \emph{preideal}
if for every $x\in X$ also every measurable function $y\colon\Omega\to E_1$
with $\Abs{y(t)}\le\Abs{x(t)}$ for almost all $t\in\Omega$ belongs
to $X$ and satisfies $\Norm y\le\Norm x$. (If $X$ is complete, it is called
an \emph{ideal} space.)

Ideal spaces are sometime also called Banach function spaces.
Foundation on their theory had been laid out independently by
W.\,A.\,J. Luxemburg and P.\,P.\ Zabrejko in a series of fundamental
papers~\cite{LuxemburgI,LuxemburgII,LuxemburgIII,LuxemburgIV,%
LuxemburgV,LuxemburgVI,LuxemburgVII,LuxemburgVIII,LuxemburgIX,LuxemburgX,%
LuxemburgXI,LuxemburgXII,LuxemburgXIII,LuxemburgXIVa,LuxemburgXIVb,%
LuxemburgXVa,LuxemburgXVb,LuxemburgXVIa,LuxemburgXVIb,ZabrejkoIdeal}
To each preideal space of vector
functions, one can associate its ``real form'' $X_{\R}$,
which is the preideal space
consisting of measurable real functions $x\colon\Omega\to\R$ such that
$x_e(t)=x(t)e$ (for some $e\in E_1$ with $\Abs e=1$) belongs to $X$,
and which satisfies $\Norm x_{X_{\R}}=\Norm{x_e}$. For example, the real form
of $L_p(\Omega,E_1)$ is $L_p(\Omega,\R)$.

The following property of a preideal space depends only on its real form:

\begin{definition}
We say that the preideal space $X$ is \emph{locally regular} if for each
$x\in X$, each set $T\subseteq\supp x$ of positive measure, and each
$\varepsilon>0$, there is a set $D\subseteq T$ of positive measure with
$\Norm{P_Dx}_X<\varepsilon$.
\end{definition}

Recall that a set $M\subseteq\Omega$ contains no atom of finite measure
if every $T\subseteq M$ with $0<\mes T<\infty$ contains a subset $D\subseteq T$
with $0<\mes D<\mes T$ (or, equivalently, $T$ can be divided into two sets
of equal measure).

\begin{proposition}\label{p:localreg}
Suppose that $\Omega$ contains no atoms of finite measure.
If a preideal space $X$ is regular, that is, for each $x\in X$ there holds
\[\lim_{\delta\to0}\,\sup_{\mes D<\delta}\,\,\Norm{P_Dx}=0\quad\text{and}\quad
\inf_{\mes D<\infty}\,\Norm{P_{\Omega\setminus D}x}=0\text,\]
then $X$ is locally regular. In particular, $X=L_p(\Omega,E_1)$ is
locally regular if $\Omega$ contains no atoms of finite measure and
$1\le p<\infty$.
\end{proposition}
\begin{proof}
Given $x\in X$, a set $T\subseteq\supp x$ of positive measure.
and $\varepsilon>0$, there exists a set $T_0\subseteq\Omega$ of finite
measure with $\Norm{P_{\Omega\setminus T_0}x}<\varepsilon$. Hence, if
$D_0\coloneqq T\setminus T_0$ has positive measure, $D=D_0$ has the
required property. Otherwise, $\mes T\le\mes T_0<\infty$. Since $T$ has
no atoms, there is a sequence $T\supseteq D_n\supseteq D_{n+1}$ with
$0<\mes D_n\le2^{-n}\mes T\to0$, and so $\Norm{P_{D_n}x}\to0$. Hence,
$D=D_n$ with sufficiently large $n$ has the required property.
\end{proof}

Even in case $\Omega=[0,1]$ the converse of Proposition~\ref{p:localreg} fails:

\begin{example}
Let $M$ be a Young function and $X=L_M([0,1],\R)$ be the corresponding
Orlicz space, see~\cite{KrasOrlicz}. Then $X$ is regular if and only if
$M$ satisfies the $\Delta_2$-condition, see
e.g.~\cite[Capter II, \S10, Section 6]{KrasOrlicz}.
However, $X$ is locally regular even if $M$ violates the $\Delta_2$-condition.

Indeed, if $T\subseteq\supp x$ has positive measure,
put $T_n=\Curly{t\in T:\Abs{x(t)}\le n}$. Since $\bigcup_nT_n=T$
has positive measure, there is some $N$ with $\mes T_N>0$.
Then $y=P_{T_N}x$ is bounded and thus has absolutely continuous norm,
see~\cite[Capter II, \S10, Section 6]{KrasOrlicz}. In particular, if
$D\subseteq T_N$ has sufficiently small measure, there holds
$\Norm{P_Dx}_X=\Norm{P_Dy}_X<\varepsilon$.
\end{example}

Now we can give a first illustration of Definition~\ref{d:Vpair}:

\begin{proposition}\label{p:simpleVpair}
Let $Y$ be a subspace of a locally regular preideal space $Z$.
Let $X$ be a preideal space with a bounded embedding $X_{\R}\subseteq Z_{\R}$.
Then $(X,Y)$ is a weak $V$-pair.
\end{proposition}
\begin{proof}
Let $x\in X$, $y\in Y$, and $T\subseteq\Omega$ be measurable with $P_Ty\ne0$.
There is a natural number $N$ such that
\[T_N\coloneqq\Curly{t\in T:\text{$N\Abs{y(t)}>\Abs{x(t)}$}}\]
has positive measure, because otherwise $\bigcup_{N=1}^\infty T_N=\supp(P_Ty)$
would be a null set. Since $Z$ is locally regular, there is a sequence
$D_n\subseteq T_N$, $\mes D_n>0$, with $\Norm{P_{D_n}y}_Y\to0$.
The definition of $T_N$ implies that for every $y_n\in P_{D_n,y}$, there holds
\[\Abs{P_{D_n}x(t)}\le N\Abs{P_{D_n}y(t)}\le N\Abs{y_n(t)}\]
for almost all $t$. It follows that
\[\Norm{\,\Abs{P_{D_n}x}\,}_{Z_{\R}}\le
N\Norm{P_{D_n}y}_Y\le N\Norm{y_n}_Y\text.\]
Putting $x_n\coloneqq P_{D_n}x$, we thus find by the continuity of the
embedding $X_{\R}\subseteq Z_{\R}$ that $\Norm{x_n}_X\to0$ and that
$\Norm{x_n}_X/\Norm{y_n}_Y$ is bounded.
\end{proof}

We can relax the hypothesis about the embedding $X_{\R}\subseteq Z_{\R}$
in Proposition~\ref{p:simpleVpair} and even the hypothesis that $X$
is a preideal space. We will apply this technical extension in a moment:

\begin{proposition}\label{p:clumsyVpair}
Let $Y$ be a subspace of a locally regular preideal space $Z$.

Then $(X,Y)$ is a weak $V$-pair whenever $X$ is a normed space with the
following property.
For each $x\in X$, $y\in Y$ and each measurable $T\subseteq\Omega$
with $P_Ty\ne0$ there is a measurable $T_0\subseteq T$ with $P_{T_0}y\ne0$ and
a constant $C>0$ such that for every measurable $D\subseteq T_0$ of positive
measure there are a subset $D_0\subseteq D$ of positive measure and
$x_D\in P_{D_0,x}\cap X$ with
$\Norm{x_D}_X\le\Norm{\,\Abs{P_{D_0x}}\,}_{Z_{\R}}$.
\end{proposition}
\begin{proof}
In the proof of Proposition~\ref{p:simpleVpair}, we just replace $T$ by $T_0$
and $D_n$ by subsets such that (by hypothesis) there are
functions $x_n\in P_{D_n,x}\cap X$ satisfying
$\Norm{x_n}_X\le C\Norm{\,\Abs{P_{D_n}x}\,}_{Z_{\R}}$.
\end{proof}

For the case that $\mes\Omega=\infty$ or that $X$ is a subspace of smooth
functions, the subsequent result does not follow from
Proposition~\ref{p:simpleVpair} alone, but it follows from
the technical extension in Proposition~\ref{p:clumsyVpair}.

\begin{corollary}
Let $\Omega$ contain no atom of finite measure, and $1\le p\le q<\infty$.
Then $(X,Y)=(L_p(\Omega,E_1),L_q(\Omega,E_2))$ is a weak $V$-pair.

The same assertion holds if $Y$ is replaced by any subspace of
$L_q(\Omega,E_2)$, endowed with the $L_q$-norm, and/or if
$\Omega$ is a regular Radon measure space without finite atoms
and $X$ is replaced by the subspace consisting of all
continuous functions of $L_p(\Omega,E_2)$, endowed with the $L_p$-norm.
If $\Omega$ is a manifold of class $C^n$ or $C^\infty$, then $X$
can also be replaced by the subset of functions from the class
$C^n$ or $C^\infty$.
\end{corollary}
\begin{proof}
Note that $Z\coloneqq L_q(\Omega,E_2)$ is a locally regular preideal space by
Proposition~\ref{p:localreg} with real form $Z_{\R}=L_q(\Omega,\R)$.
For the case $\mes\Omega<\infty$ and $X=L_p(\Omega,E_1)$, the assertion
follows from Propositionn~\ref{p:simpleVpair}. The general case follows
from Proposition~\ref{p:clumsyVpair}:

If $y\in Y$ then $\supp Y$ is $\sigma$-finite. Hence, if $P_Ty\ne0$ then
$T\cap\supp y$ contains a subset $T_0$ of positive finite measure, and
so we have a continuous embedding of $L_p(T_0,\R))$ into $L_q(T_0,\R)$.
In the case $X=L_p(\Omega,E_1)$, we thus have the property of
Proposition~\ref{p:clumsyVpair} with $x_D\coloneqq P_Dx$ and $D_0\coloneqq D$.

For the case that $\Omega$ is a regular Radon measure,
we can additionally choose $D_0$ in Proposition~\ref{p:clumsyVpair}
to be compact. To find $x_D$ in Proposition~\ref{p:clumsyVpair},
we first choose a sequence of open sets $U_k\supseteq D_0$
with $\mes(U_k\setminus D_0)\to0$. By Urysohn's lemma, there are
continuous functions $\lambda_k\colon\Omega\to[0,1]$ satisfying
$\lambda_k(t)=1$ for $t\in D_0$ and $\lambda_k(t)=0$ for $t\notin U_k$.
Moreover, $\lambda_k$ can be chosen to be of class $C^n$ or $C^\infty$
if $\Omega$ is a manifold of class $C^n$ or $C^\infty$, respectively,
see e.g.~\cite[Theorem~9.8]{VaethTriple}. In view of $x\in X$
and since $X$ is endowed with the norm of $L_p(X,E_1)$, the function sequence
$x_k(t)\coloneqq\lambda_k(t)x(t)$ thus belongs to $P_{D_0,x}\cap X$
and satisfies $\Norm{x_k}\to\Norm{P_{D_0}x}$ as $k\to\infty$.
For sufficiently large $k$, the function $x_D\coloneqq x_k$ thus has the
property required in Proposition~\ref{p:clumsyVpair}.
\end{proof}

Now we explain the relation of Definition~\ref{d:Vpair} to the notion of
a $V$-pair introduced in~\cite[Section~2.6]{Appell}.
We first recall the latter definition.

If $\mes\Omega<\infty$ and $\Omega$ has no atoms of finite measure then
a pair of ideal spaces $(X,Y)$ of functions on $\Omega$ is a \emph{$V$-pair}
if there are $u_0\in X$ and measurable $v_0\colon\Omega\to\R$
with $\supp u_0=\supp v_0=\Omega$ and a constant $C$ with
\begin{equation}\label{e:Vpair}
\Norm{P_Du_0}_X\sup_{\Norm y_Y\le1}\,\int_D\Abs{v_0(t)y(t)}\dt\le C\mes D
\end{equation}
for every measurable $D\subseteq\Omega$.

\begin{proposition}\label{p:Vpair}
Every $V$-pair $(X,Y)$ is a weak $V$-pair.
\end{proposition}
\begin{proof}
Let $x\in X$, $y\in Y$ and $T$ with $P_Ty\ne0$ be as in
Definition~\ref{d:Vpair}. There is a natural number $N$ such that
\[T_N\coloneqq\Curly{t\in T:\text{$N\Abs{u_0(t)}\ge\Abs{x(t)}$ and
$N\Abs{v_0(t)y(t)}\ge1$}}\]
has positive measure, because otherwise $\bigcup_{N=1}^\infty T_N=T\cap\supp y$
would be a null set (recall that $\supp u_0=\supp v_0=\Omega$).
Since $\Omega$ has no atoms, there is a sequence $D_n\subseteq T_N$
with $0<\mes D_n\to0$. We claim that $x_n\coloneqq P_{D_n}x$ has the property
required in Definition~\ref{d:Vpair}. Indeed, if $y_n\in P_{D_n,y}\cap Y$,
then $c_n\coloneqq1/\Norm{y_n}_Y$ satisfies $\Norm{c_nP_{D_n}y}_Y\le1$.
Hence,~\eqref{e:Vpair} implies
\[\frac{\Norm{x_n}_X}{\Norm{y_n}_Y}=
\frac{\Norm{x_n}_X}{\mes D_n}\int_{D_n}c_n\dt\le
\frac{N\Norm{P_{D_n}u_0}_X}{\mes D_n}
\int_{D_n}N\Abs{v_0(t)c_nP_{D_n}y(t)}\dt\le
N^2C\text.\]
The constant $N^2C$ is independent of $n$ and of the choice of $y_n$.
\end{proof}

\section{Differentiability}\label{s:diff}

\begin{definition}\label{d:maximal}
An operator $G_0\colon U\to Y$ with $U\subseteq X$ is \emph{locally determined}
if for every measurable $D\subseteq X$ and every $x\in U$ the restriction
$G_0(x_D)|_D$ is (almost everywhere) independent of the choice
$x_D\in P_{D,x}\cap U$. We say that such a map $G_0$ has
\emph{maximal domain} if the following property implies $x\in U$:
$x\in X$, and there is some $y\in Y$ such that every set
$T\subseteq\Omega$ of positive measure contains a subset $D\subseteq T$ of
positive measure such that there is $x_D\in P_{D,x}\cap U$ with
$y|_D=G_0(x_D)|_D$ (and in this case necessarily $G_0(x)=y$).
\end{definition}

\begin{definition}
An operator $F\colon U\to Y$ with $U\subseteq X$ is a
\emph{$G$-abstract fully nonlinear integral operator} at $x_0\in U$
(with maximal domain) if $G$ is a map $G\colon M\to Y$ with
$M\subseteq X\times X$ containing the diagonal $\Curly{(x,x):x\in U}$
such that $F(x)=G(x,x)$ for all $x\in M$ and such that
$G(\parameter,x_0)$ is locally determined (with maximal domain).
\end{definition}

When we speak about fully nonlinear integral operators,
we always have a natural such map $G$ in mind:

\begin{proposition}\label{p:abstract}
Let $X$ and $Y$ be spaces of measurable functions.
Let $D_0\subseteq\Omega\times E_1$, $D_1\subseteq\Omega\times\Omega\times E_1$,
$k_i\colon D_i\to E_2$ $(i=0,1)$,
$D\subseteq\Omega\times\Omega\times E_1\times E_1$, and $k\colon D\to E_2$.
Define $K_0$ and $K_1$ by the formulas~\eqref{e:Kzero} and~\eqref{e:Kone},
and $K_2$ by
\[K_2(x_1,x_2)(t)=\int_\Omega k(t,s,x_1(t),x_2(s))\ds\text.\]
Let
\[G(x_1,x_2)\coloneqq K_0(x_1)+K_1(x_2)+K_2(x_1,x_2)\text.\]
Then $G(\parameter,x_0)$ is locally determined for every $x_0\in X$.
Moreover, we choose the natural domain $M_{x_0}\subseteq X$
of definition of $G(\parameter,x_0)\colon M_{x_0}\to Y$, that is,
if $M_{x_0}$ contains all those $x_1\in X$ such that the function
$G(x_1,x_0)$ is defined almost everywhere and an element of $Y$, then
$G(\parameter,x_0)$ has maximal domain $M_{x_0}$.

The generalized fully nonlinear integral operator
$F=K_0+K_1+K$ from the introduction satisfies $F(x)=G(x,x)$ and is a
$G$-abstract fully nonlinear integral operator  (with maximal domain)
at every point $x_0$ from the domain of $F$.
\end{proposition}
\begin{proof}
Since $K(x)=K_2(x,x)$, the only nontrivial assertion is that
$G(\parameter,x_0)$ is locally determined with maximal domain $U=M_{x_0}$.

Let $x\in X$ and $y\in Y$ be as in Definition~\ref{d:maximal}. We are to show
that $x\in U$ and $G(x,x_0)=y$. Assume by contradicion that there is a set
$T\subseteq\Omega$ of positive measure such that for almost all $t\in T$
the value $G(x,x_0)(t)$ is undefined or differs from $y(t)$.
By hypothesis, there is a subset $D\subseteq T$ of positive measure and
$x_D\in P_{D,x}\cap U$ such that $y|_D=G(x_D,x_0)|_D$, and so the form of $G$
implies that $G(x,x_0)(t)$ is defined and equal to $G(x_D,x_0)(t)=y(t)$
for almost all $t\in D$, which is a contradiction.
\end{proof}

We point out that it plays no role for the assertion (and proof) of
Proposition~\ref{p:abstract} whether we understand the integral in $K_1$
and $K_2$ in the sense of Bochner or Pettis or in some other sense
(Kurzweil, etc.): The only property which we require from the used integration
theory is that the existence and value of the integral depends only on the
equivalence class of the integrand.

In order to formulate our degeneration result for differentiable
($G$-abstract) fully nonlinear integral operators, we need an
auxiliary notion of differentiability.

If $X_1$ and $X_2$ are normed spaces, we denote the set of all linear bounded
operators $L\colon X_1\to X_2$ by $\Lc(X_1,X_2)$.

\begin{definition}
Let $X_1,X_2,Y$ be normed spaces, and $M\subseteq X_1\times X_1$.
We call a function $G\colon M\to Y$ \emph{diagonal-differentiable}
(with respect to the second variable) at an interior point
$(x_0,y_0)\in X_1\times X_2$ of $M$ if there is $L\in\Lc(X_2,Y)$ with
\[\lim_{\Norm h\to0}\,\sup_{\Norm k\le\Norm h}\,
\frac{\Norm{G(x_0+k,y_0+h)-G(x_0+k,y_0)-Lh}}{\Norm h}=0\text.\]
In this case we call $L$ the (partial) diagonal-derivative.
\end{definition}

Taking only $k=0$ in this limit, we obtain as a trivial special case:

\begin{proposition}\label{p:diagfrechet}
If $G$ is diagonal-differentiable with respect to the second variable
at $(x_0,y_0)$, then the partial Fr\'{e}chet derivative (hence also
Gateaux derivative) $D_2G(x_0,y_0)$ with respect to the second variable exists
and is equal to the diagonal-derivative $L$.
In particular, $L=D_2G(x_0,y_0)$ is uniquely determined.
\end{proposition}

The following result shows that being diagonal-differentiable (with respect
to the second variable) is actually only mildly more a requirement than
being partially differentiable (with respect to the second variable):

\begin{proposition}\label{p:diagonal}
Let $X_1,X_2,Y$ be normed spaces, $(x_0,y_0)$ be an interior point of
$U_1\times U_2\subseteq X_1\times Y_1$, and $G\colon U\times V\to Z$.
Then each of the following hypotheses implies that $G$ is
diagonal-differentiable with respect to the second variable at $(x_0,y_0)$
with partial diagonal-derivative $L$:
\begin{enumerate}
\item\label{i:contderiv}
The partial Gateaux derivatives $D_2G(x,y)$ exist
for all $(x,y)$ in a neighborhood of $(x_0,y_0)$, and
\[\lim_{(x,y)\to(x_0,y_0)}D_2G(x,y)=D_2G(x_0,y_0)=L\]
in operator norm.
\item\label{i:normcont}
There is a neighborhood $U_0\subseteq U_1$ of $x_0$ such that the
family of functions $G(x,\parameter)$ $(x\in U_0)$ is
equidifferentiable at $y_0$ with derivatives $D_2G(x,y_0)$, that is
\[\lim_{\Norm h\to0}\,\sup_{x\in U_0}\,
\frac{\Norm{G(x,y_0+h)-G(x,y_0)-D_2G(x,y_0)h}}{\Norm h}=0\text,\]
and
\[\lim_{x\to x_0}D_2G(x,y_0)=D_2G(x_0,y_0)=L\]
in operator norm.
\item\label{i:independent}
$G$ is independent from the second variable and $L=0$.
\end{enumerate}
\end{proposition}
\begin{proof}
Suppose that~\iref{i:contderiv} holds. Then for each $\varepsilon>0$ there
is $\delta>0$ with $\Norm{D_2G(x,y)-L}<\varepsilon$ for every
$x=x_0+k$ with $\Norm k\le\delta$ and $\Norm{y-y_0}\le\delta$.
Applying Lemma~\ref{l:mean} from the appendix with
$F_{x,h}(\lambda)=G(x,y_0+\lambda h)-G(x,y_0)-L\lambda h$ and
$0<\Norm h\le\delta$, we obtain
\begin{gather*}
\Norm{G(x,y_0+h)-G(x,y_0)-Lh}=\Norm{F_{x,h}(1)-F_{x,h}(0)}\le
\sup_{\lambda\in[0,1]}\,\Norm{F_{x,h}'(\lambda)}=\\
\sup_{\lambda\in[0,1]}\,\Norm{D_2G(x,y_0+\lambda h)h-Lh}\le
\varepsilon\Norm h\text.
\end{gather*}
Dividing by $\Norm h$, we obtain the assertion.

If~\iref{i:normcont} holds, then for each $\varepsilon>0$ there
is $\delta>0$ such that $\Norm{D_2G(x,y_0)-L}<\varepsilon$ whenever
$x=x_0+k$ with $\Norm k\le\delta$. In particular, if
$\Norm k\le\Norm h\le\delta$, we have
\[\frac{\Norm{G(x,y_0+h)-G(x,y_0)-Lh}}{\Norm h}\le
\frac{\Norm{G(x,y_0+h)-G(x,y_0)-D_2G(x,y_0)h}}{\Norm h}+\varepsilon
\text.\]
Now the claim follows from the equidifferentiability.

The assertion~\iref{i:independent} is trivial (and also follows from each
of the previous assertions).
\end{proof}

\begin{remark}\label{r:contFrechet}
As a side result, we re-obtain from
Proposition~\ref{p:diagonal}\iref{i:contderiv} and
Proposition~\ref{p:diagfrechet} the well-known fact that if the
Gateaux derivatives $D_2G$ are continuous in some point $(x_0,y_0)$
(in the interior of their domain of definition) then $D_2G(x_0,y_0)$
is necessarily even a Fr\'{e}chet derivative.
\end{remark}

Now we are in a position to formulate the general form of our
announced degeneracy result concerning differentiability.

We call a map $F\colon X\to Y$ \emph{bounded affine}, if it has the form
$F(x)=y+Ax$ with some $y\in Y$ and $A\in\Lc(X,Y)$.

\begin{theorem}[Differentiability-Degeneration]\label{t:fulldiff}
Let $(X,Y)$ be a weak $V$-pair, $U\subseteq X$, and
$F\colon U\to Y$ be a $G$-abstract fully nonlinear integral operator at
$x_0\in U$ with maximal domain. Suppose in addition that $G$ is
diagonal-differentiable at $(x_0,x_0)$ with respect to the second variable.

If $F$ is Fr\'{e}chet differentiable at $x_0$ then $U=X$, and
$G(\parameter,x_0)$ is bounded affine.
\end{theorem}

Actually, Theorem~\ref{t:fulldiff} holds also for any other class of spaces
and operators for which a degreneration results for the corresponding operator
$G_0=G(\parameter,x_0)$ is available. In fact, the main idea of the proof is
to show first that $G_0$ is Fr\'{e}chet differentiable at $x_0$.
Since this assertion is of independent interest, let us formulate
this part more general:

\begin{lemma}\label{l:frechdiff}
Let $X,Y$ be normed spaces, and $M\subseteq X\times X$.
If $G\colon M\to Y$ is diagonal-differentiable at $(x_0,x_0)$ with
diagonal derivative $D_2G(x_0,x_0)$
and $F(x)=G(x,x)$ is Fr\'{e}chet differentiable at $x_0$, then also
$G(\parameter,x_0)$ is Fr\'{e}chet differentiable at $x_0$ with derivative
\begin{equation}\label{e:diagdiff}
D_1G(x_0,x_0)=DF(x_0)-D_2G(x_0,x_0)\text.
\end{equation}
\end{lemma}
\begin{proof}
Putting $L_2=D_2G(x_0,x_0)$, $L=DF(x_0)$, and $L_1\coloneqq L-L_2$,
we have for all $h\in X$ with sufficiently small $\Norm h>0$ that
\begin{gather*}
G(x_0+h,x_0)-G(x_0,x_0)-L_1h=\bigl(F(x_0+h)-F(x_0)-Lh\bigr)-\\
\bigl(G(x_0+h,x_0+h)-G(x_0+h,x_0)-L_2h\bigr)\text.
\end{gather*}
Dividing this equation by $\Norm h$ and letting $h\to0$, we obtain by
definition of the Fr\'{e}chet derivative $L$ and by definition of the
diagonal-derivative $L_2$ that $G_0=G(\parameter,x_0)$ has in $x_0$ the
Fr\'{e}chet derivative $D_1G(x_0,x_0)=L_1$.
\end{proof}

The formula~\eqref{e:diagdiff} is not surprising since it would
follow from the chain rule \emph{if} one would know that $G$ is
Fr\'{e}chet differentiable. The crucial point of Lemma~\ref{l:frechdiff}
is that the latter is not assumed.

\begin{proof}[Proof of Theorem~\ref{t:fulldiff}]
By Lemma~\ref{l:frechdiff}, we know that $G_0=G(\parameter,x_0)$
has in $x_0$ a Fr\'{e}chet derivative $L_1$.
We have to show that $G_0$ is actually bounded affine.
Since $G_0$ is locally determined with maximal domain $U$, and $L_1$ is
in particular the Gateaux derivative of $G_0$, it follows that also $L_1$
is locally determined (with domain $X$). Hence, the map
$G_1(x)\coloneqq G_0(x+x_0)-G_0(x_0)-L_1x$ is locally determined with
maximal domain $U_1=U-x_0$, and Fr\'{e}chet differentiable at the interior
point $0$ of $0$ with $G_1(0)=0$ and $DG_1(0)=0$.
We are to show that every $x\in X$ belongs to $U_1$ and satisfies $G_1(x)=0$.

Since $G_1$ is locally determined with maximal domain $U_1$,
it suffices to show that every set $T\subseteq\Omega$ of positive measure
contains a subset $D\subseteq T$ of positive measure such that there is
$x_D\in P_{D,x}\cap U_1$ with $G_1(x_D)|_D=0$.

To see this, we note first that $x_0$ is an interior point of $U$ and thus
$0$ is an interior point of $U_1$. Hence, by Definition~\ref{d:Vpair},
there is a set $D\subseteq T$ of positive measure such that there is some
$x_D\in P_{D,x}$ which belongs to $U_1$.
Replacing $x$ by $x_D$ and $T$ by $D$, we thus can assume without loss of
generality that $x\in U_1$ and thus $y\coloneqq G_1(x)\in Y$.

Now assume by contradiction that there is a subset $T\subseteq\Omega$ of
positive measure with $y|_T\ne0$. Let $D_n$ and $x_n$ be as in
Definition~\ref{d:Vpair}, and let
\[C>\liminf_{n\to\infty}\,\sup_{y_n\in P_{D_n,y}\cap Y}\,
\frac{\Norm{x_n}_X}{\Norm{y_n}_Y}\text.\]
Since $G_1$ is Fr\'{e}chet differentiable at $0$ with $G_1(0)=0$ and
$DG_1(0)=0$, we have for all sufficiently large $n$ that $x_n\in U_1$,
that is $G_1(x_n)\in Y$, and
\[\Norm{G_1(x_n)}_Y\le C^{-1}\Norm{x_n}_X\text.\]
Since $G_1$ is a local operator, we have $y_n\coloneqq G_1(x_n)\in P_{D_n,y}$,
and our choice of $C$ thus implies that there is some large $n$ with
\[C\Norm{y_n}_Y>\Norm{x_n}_X\text,\]
which is a contradiction.
\end{proof}

Theorem~\ref{t:fulldiff} contains all folklore results about the
degeneration of differentiable superposition operators in $L_p$-spaces
and generalizations thereof,
see e.g.~\cite{AppellDegeneration,Appell,KrasTop,WangSupDiff,VainbergSup}.
We point out once more that, in contrast to these results, we cover the
case of vector-valued functions and if $\Omega$ fails to have finite measure
or even fails to be $\sigma$-finite:

\begin{corollary}[Special Case of Superposition Operators]\label{c:diff}
Let $(X,Y)$ be a weak $V$-pair, e.g.\ $X=L_p(\Omega,E_1)$ and
$Y=L_q(\Omega,E_2)$
where $\Omega$ has no atoms of finite measure and $1\le p\le q<\infty$.
If $F(x)(t)=f(t,x(t))$ acts
from $U\subseteq X$ into $Y$ and is Fr\'{e}chet differentiable in some
$x_0\in U$, then $F$ acts from $X$ into $Y$ and is bounded affine.
\end{corollary}
\begin{proof}
In view of Proposition~\ref{p:abstract}, $F(x)=G(x,x)$ is an
$G$-abstract fully nonlinear integral operator with full domain of
definition at every point $x\in X$, where $G(x_1,x_2)=K_0(x_1)$ is independent
of $x_2$. Proposition~\ref{p:diagonal}\iref{i:independent} implies
that $G$ is diagonal-differentiable with respect to the second variable,
so that all hypotheses of Theorem~\ref{t:fulldiff} are satisfied.
\end{proof}

\begin{remark}\label{r:diagonal}
It is clear that the hypothesis about diagonal-differentiable $G$ cannot be
dropped in Theorem~\ref{t:fulldiff}. In fact, even the zero function can
be written as a $G$-abstract fully nonlinear integral operator
with non-differentiable (hence, not bounded affine) $G(\parameter,x)$.
By means of examples, let $G_0$ be a non-differentiable
superposition operator, and $G(x_1,x_2)=G_0(x_1)-G_0(x_2)$ or
$G(x_1,x_2)=G_0(x_1-x_2)-G_0(0)$.
\end{remark}

The hypothesis about diagonal-differentiable $G$ in the abstract
Theorem~\ref{t:fulldiff} can be verified for particular fully nonlinear
integral operators by using Proposition~\ref{p:diagonal}. Note that for
obtaining the required partial derivative with respect to the second variable,
one can use classical criteria for differentiability of Urysohn operators,
see e.g.~\cite[Theorem~20.5--20.9]{Krasnoselskii}.
As an example, we formulate only a simple special case which can also deal with
vector-valued functions and holds also if $\mes\Omega=\infty$.

\begin{theorem}[Example of Differentiability-Degeneration]\label{t:fulldiffx}
Let $\Omega$ be $\sigma$-finite.
Let $1<p\le q<\infty$, $X=L_p(\Omega,E_1)$, and $Y=L_q(\Omega,E_2)$.
Let $k_0\colon\Omega\times E_1\to E_2$,
$k_1\colon\Omega\times\Omega\times E_1\to E_2$,
and $k\colon\Omega\times\Omega\times E_1\times E_1\to E_2$
satisfy a Carath\'{e}odory condition, that is,
be measurable with respect to the measure-space argument
and continuous with respect to the Banach-space argument. Suppose also that the
partial Gateaux derivatives $D_3k_1$ and $D_4k$ exist and satisfy a
Carath\'{e}odory condition (as functions with values in $\Lc(E_1,E_2)$)
and the joint growth estimate
\[\Norm{D_3k_1(t,s,v)+D_4k(t,s,u,v)}_{\Lc(E_1,E_2)}\le
a(t)\bigl(b(s)+c\Abs v^{p-1}\bigr)\]
where $a\in L_q(\Omega,[0,\infty))$, $b\in L_{p/(p-1)}(\Omega,[0,\infty])$,
$c\in[0,\infty)$. If
\[G(x_1,x_2)(t)\coloneqq k_0(t,x_1(t))+\int_\Omega k_1(t,s,x_2(s))\ds+
\int_\Omega k(t,s,x_1(t),x_2(t))\ds\]
maps an open subset $M\subseteq X\times X$ into $Y$,
then $G$ is partially Fr\'{e}chet differentiable
with respect to the second variable at each $(x_1,x_2)\in M$,
the derivative
\begin{equation}\label{e:Urydiff}
D_2G(x_1,x_2)h(t)\coloneqq\int_\Omega D_4k(t,s,x_1(t),x_2(t))h(s)\ds\text.
\end{equation}
being continuous. Moreover, if the map
$F(x)=G(x,x)$ is Fr\'{e}chet differentiable at some point $x_0$, then
$G(\parameter,x_0)$ acts by the above formula from $X$ into $Y$ and
is bounded affine.
\end{theorem}
\begin{proof}
In order to show that $D_2G$ exists and is continuous, we do not need to
consider the first summand, and we can combine the last two; hence,
we assume first without loss of generality that $k_0=0$ and $k_1=0$.

Let $x_1,x_2,h\in X$ be fixed. For $\lambda\in\R\setminus\Curly0$,
we have by our assumption
\begin{gather*}
D(\lambda)(t)=\frac{G(x_1,x_2+\lambda h)(t)-G(x_1,x_2)(t)}{\lambda}=\\
\int_\Omega\frac{k(t,s,x_1(t),x_2(s)+\lambda h(s))-
k_1(t,s,x_1(t),x_2(s))}{\lambda}\ds\text.
\end{gather*}
By Lemma~\ref{l:mean} from the appendix, the integrand is bounded by
\[\Norm{D_4k(t,s,x_1(t),x_2(s))}_{\Lc(E_1,E_2)}\Abs{h(s)}\le y_t(s)\coloneqq
a(t)\bigl(b(s)+c\Abs{x_2(s)}^{p/p'}\bigr)\Abs{h(s)}\text,\]
where $\frac1p+\frac1{p'}=1$.
By H\"{o}lder's inequality, $y_t$ is integrable for almost all $t\in\Omega$,
and so Lebesgues dominated convergence theorem implies
that $D(\lambda)(t)\to D_2G(x_1,x_2)h(t)$ for almost all $t\in\Omega$, where
$D_2G(x_1,x_2)h$ is defined by~\eqref{e:Urydiff}.
Moreover, since $z(t)=\int_\Omega y_t(s)\ds$ is bounded
by some multiple (independent of $\lambda$) of $a$, a further application of
Lebesgue's dominated convergence theorem shows that
$D(\lambda)\to D_2G(x_1,x_2)h$ in $Y$. Hence, $G$ is Gateau differentiable
with respect to the second variable with derivative~\eqref{e:Urydiff}.

To prove that $D_2G$ is continuous, assume that $(x_{1,n},x_{2,n})\to(x_1,x_2)$
in $X\times X$. We have to show that
$D_2G(x_{1,n},x_{2,n})\to D_2G(x_1,x_2)$ in operator norm.
It suffices to show that each subsequence contains another subsequence for
which this convergence holds. Hence, passing to a subsequence, we can assume
without loss of generality that $x_{i,n}(t)\to x_i(t)$ $(i=1,2)$ for almost
all $t\in\Omega$. By Vitali's convergence theorem (see
e.g.~\cite[Theorem~1.21]{VaethIntegral}), the families $x_{i,n}$ $(i=1,2)$
have equicontinuous norm in $X$.

By H\"{o}lder's inequality, $\Abs{D_2G(x_{1,n},x_{2,n})h(t)-D_2G(x_1,x_2)h(t)}$
is, for almost all $t$, for all $h\in X$ with $\Norm h_X\le1$ bounded by
\begin{equation}\label{e:zn}
z_n(t)=\Bigl(\int_\Omega\Abs{D_4k(t,s,x_{1,n}(t),x_{2,n}(s))-
D_4k(t,s,x_1(t),x_2(s))}^{p'}\ds\Bigr)^{1/p'}\text.
\end{equation}
We are to show that $z_n\to0$ in $L_q(\Omega)$. Since $x_{i,n}\to x_i$
almost everywhere and since $D_4k$ is a Carath\'{e}odory function,
the integrand converges to zero almost everywhere. Moreover, the integrand
is bounded by
\[v_n(t,s)=2^{p'}a(t)^{p'}
\bigl(b(s)^{p'}+c\Abs{x_2(s)}^p+c\Abs{x_{2,n}(s)}^p\bigr)\text.\]
Since $x_{2,n}$ has equicontinuous norm in $X$, it follows that
$v_n(t,\parameter)$ has equicontinuous norm in $L_1(\Omega)$ for almost all
$t\in\Omega$, and so $z_n(t)\to0$ for almost all $t\in\Omega$ by
Vitali's convergence theorem. Moreover,
\[\Abs{z_n(t)}\le\Bigl(\int_\Omega v_n(t,s)\ds\Bigr)^{1/p'}\]
is uniformly bounded by some multiple of $a$, and so Lebesgue's dominated
convergence theorem implies $z_n\to0$ in $L_q(\Omega)$. Hence, $D_2G$ is indeed
continuous in operator norm on $X\times X$.

By Remark~\ref{r:contFrechet} and Proposition~\ref{p:diagonal}, we find
that $D_2G$ is the Fr\'{e}chet derivative and that $G$ is
diagonal-differentiable.
In view of Proposition~\ref{p:abstract}, the last assertion of
Theorem~\ref{t:fulldiffx} follows from Theorem~\ref{t:fulldiff}.
\end{proof}

In particular, we obtain a degeneration result for the operator
occuring in the continuous limits of coupled Kuramoto oscillators,
see e.g.~\cite{Medvedev2}.

\begin{example}\label{x:kuramoto}
Let $\Omega$ be $\sigma$-finite. For $1<p\le q<\infty$, let
\[F(x)(t)=\int_\Omega k(t,s)\sin\bigl(x(t)-x(s)\bigr)\ds\text,\]
where $k\colon\Omega\times\Omega\to\R$ is measurable and satisfies the estimate
$\Abs{k(t,s)}\le a(t)b(s)$ with some $a\in L_q(\Omega)$,
$b\in L_{p/(p-1)}(\Omega)$.
If $k\ne0$ then $F\colon L_p(\Omega)\to L_q(\Omega)$ fails to be Fr\'{e}chet
differentiable at every point $x_0\in L_p(\Omega)$.

Indeed, if $k\ne0$ then
$G(x,x_0)(t)=\int_\Omega k(t,s)\sin\bigl(x(t)-x_0(s)\bigr)\ds$
is not linear with respect to $x$, so that the assertion follows from
Theorem~\ref{t:fulldiffx}.
\end{example}

Example~\ref{x:kuramoto} shows in particular that
assertions about stability of stationary solutions for differential equations
involving such operators in $L_p(\Omega)$ as in~\cite{Medvedev2} can
\emph{not} be obtained directly by studying the spectrum/eigenvalues
of the Fr\'{e}chet derivative, because the Fr\'{e}chet derivative
does not exist at all!

\section{Local Lipschitz Condition}

For the same reason as described in Remark~\ref{r:diagonal}, we need an
additional condition to formulate a reasonable degeneration result concerning
a local Lipschitz condition for ($G$-abstract)
fully nonlinear integral operators.
It seems that the best which can be done in general, is to
impose the following hypothesis. (We will discuss in a moment that this
hypothesis holds for generalized fully nonlinear integral operators
under some structural assumptions.)

\begin{definition}\label{d:Lip}
Let $X_1,X_2,Y$ be normed spaces, and $G\colon M\to Y$ for some
$M\subseteq X_1,X_2$. We call $G$
\emph{locally $(\tau,\ell)$-Lipschitz transferrable} at some interior point
$(x_1,x_2)$ of $M$ with constants $\tau,\ell\ge0$ if there is $r>0$ such
that for all $x,y\in X_1$ and $h\in X_2$ with
$\Norm h\le\min\Curly{\Norm{x-x_1},\Norm{y-x_1}}<r$ there holds
\[\Norm{G(x,x_2+h)-G(y,x_2+h)}\le\tau\Norm{G(x,x_2)-G(y,x_2)}+
\ell\Norm{x-y}\text.\]
\end{definition}

Clearly, this property is satisfied if $G$ satisfies a local
Lipschitz condition with respect to the first argument
(with constant at most $\ell$) in a neighborhood of $(x_1,x_2)$.

However, we will show now that we do not need such a Lipschitz condition if
$G$ satisfies a certain structural condition. More precisely, it suffices
that $G$ is built from operators with ``separated variables''
in a sense. Moreover, in some cases the following result also shows that
sums of such functions are admissible in some cases.

\begin{proposition}\label{p:Liptrans}
Let $X_1,X_2,Y$ be normed spaces. Suppose that $x_i\in X_i$ $(i=1,2)$
have neighborhoods $U_i\subseteq X_i$ such that
$U_1\times U_2\subseteq M\subseteq X_1\times X_2$, $G\colon M\to Y$,
and $\tau,\ell\ge0$.
\begin{enumerate}
\item\label{i:lt1}
If $G|_{U_1\times U_2}$ has the form $G(y_1,y_2)=G_1(y_1)$ with arbitrary
$G_1\colon U_1\to Y$ then $G$ is locally $(1,0)$-Lipschitz transferrable
at $(x_1,x_2)$.
\item\label{i:lt2}
If $G|_{U_1\times U_2}$ has the form $G(y_1,y_2)=G_0(y_1,y_2)+G_2(y_2)$
with arbitrary $G_2\colon U_2\to Y$, and $G_0$ is locally
$(\tau,\ell)$-Lipschitz transferrable at $(x_1,x_2)$ then also
$G$ is locally $(\tau,\ell)$-Lipschitz transferrable at $(x_1,x_2)$.
\item\label{i:ltL}
Suppose that $G|_{U_1\times U_2}$ has the form
\[G(y_1,y_2)=G_0(y_1,y_2)+L(y_2)G_1(y_1)+G_2(y_2)\text,\]
where $G_0(\parameter,y_2)$ satisfies a Lipschitz condition on $U_1$ with
constant at most $\ell$ for every $y_2\in U_2$, $G_2\colon U_2\to Y$ is
abitrary, and $G_1\colon U_1\to Y_1$ is abitrary with some normed space $Y_1$.
Moreover, suppose that for every $y_2\in U_2$ the map
$L(y_2)\colon Y_1\to Y$ is linear and satisfies
\begin{equation}\label{e:ltL}
\Norm{L(y_2)y}_Y\le\tau\Norm{L(x_2)y}_Y\quad\text{for all $y\in Y_1$.}
\end{equation}
In the case $\tau>0$ suppose in addition that
\begin{equation}\label{e:ltrange}
G_0(y_1,x_2)\in L(x_2)(Y_1)\quad\text{for every $y_1\in U_1$.}
\end{equation}
Then $G$ is locally $(\tau,(1+\tau)\ell)$-Lipschitz transferrable
at $(x_1,x_2)$.
\end{enumerate}
\end{proposition}
\begin{proof}
Assertions~\iref{i:lt1} and~\iref{i:lt2} are immediate from the definition.
For the proof of~\iref{i:ltL}, we can assume $G_2=0$ by~\iref{i:lt2}.
Let $h\in X_2$ be such that $x_2+j\in U_2$. Then~\eqref{e:ltL} implies
that the null space of $L(x_2)$ is contained in the null space of
$L(x_2+h)$. Hence, there is a unique linear map $M_h$ from the subspace
$Y_0=L(x_2)(Y_1)\subseteq Y$ into $Y$ satisfying
$M_hL(x_2)y=L(x_2+h)y$ for all $y\in Y_1$. By~\eqref{e:ltL}, this map has
operator norm at most $\tau$. Now we calculate for every $x,y\in U_1$ that
\begin{gather*}
G(x,x_2+h)-G(y,x_2+h)=M_h\bigl(G(x,x_2)-G(y,x_2)\bigr)+{}\\
\bigl(G_0(x,x_2+h)-G_0(y,x_2+h)\bigr)-M_h\bigl(G_0(x,x_2)-G_0(y,x_2)\bigr)
\text,
\end{gather*}
where in case $\tau>0$ this calculation is valid by~\eqref{e:ltrange}.
(In case $\tau=0$, we can extend $M_h=0$ to the whole space $Y$ and do not
need such a requirement.)
Taking the norms and applying the triangle inequality, we obtain
\[\Norm{G(x,x_2+h)-G(y,x_2+h)}\le
\tau\Norm{G(x,x_2)-G(y,x_2)}+\ell\Norm{x-y}+\tau\ell\Norm{x-y}\text,\]
and so the assertion follows.
\end{proof}

Let us show now that the above structural hypothesis is satisfied for
general fully nonlinear integral operators if the kernel function $k$
has a certain structure.

A normed space $Y$ of measurable functions $y\colon\Omega\to E_2$ is called
\emph{preideal$^*$} if for every $y\in Y$ and every $x\in L_\infty(\Omega,\R)$
the function $(xy)(t)\coloneqq x(t)y(t)$ belongs to $Y$ and satisfies
$\Norm{xy}_Y\le\Norm x_{L_\infty}\Norm y_Y$.
Every preideal space is a preideal$^*$ space,
but the converse holds in general only if $E_2=\R$. Important examples
of preideal$^*$ spaces which fail to be preideal are Orlicz spaces generated
by a non-radial Young function, see e.g.~\cite[Example~2.1.1]{VaethIdeal}.

\begin{corollary}\label{c:Liptrans}
Let $X_i$ $(i=1,2)$ be spaces of measurable functions,
and $Y$ be a preideal$^*$ space.
Let $x_i\in X$ contain neighborhoods $U_i\subseteq X_i$ such that
$K_0$ and $K_1$, e.g.\ defined by the formulas~\eqref{e:Kzero}
and~\eqref{e:Kone}, act from $U_1$ or $U_2$ into $Y$, respectively,
and that $K_0$ satisfies a Lipschitz condition on $U_1$ with constant $\ell$.
Moreover, let
\[K_2(y_1,y_2)(t)=\int_\Omega g_1(t,y_1(t))g_2(t,s,y_2(s))\ds\text,\]
where
\[G_1(y_1)(t)=g_1(t,y_1(t))\]
acts from $U_1$ into $Y$, and the Urysohn operator
\[G_2(y_2)(t)=\int_\Omega g_2(t,s,y_2(s))\ds\]
acts from $U_2$ into $L_\infty(\Omega,\R)$.
Suppse that there is $\tau\ge0$ satisfying
\begin{equation}\label{e:ltG2}
\Abs{G_2(y_2)(t)}\le\tau\Abs{G_2(x_2)(t)}\quad\text{for almost all $t$}
\end{equation}
for every $y_2\in U_2$. Moreover, suppose that for every $y_1\in U_1$
the function
\begin{equation}\label{e:preimage}
t\mapsto K_0(y_1)(t)/G_2(x)(t)\quad\text{(putting $0/0\coloneqq0$)}
\end{equation}
is almost everywhere finite and belongs to $Y$.

Then $G(y_1,y_2)=K_0(y_1)+K_1(y_2)+K_2(y_1,y_2)$ is locally
$(\tau,(1+\tau)\ell)$-Lipschitz transferrable at $(x_1,x_2)$.
\end{corollary}
\begin{proof}
The assertion follows from Proposition~\ref{p:Liptrans}\iref{i:ltL} with
$G_0=K_0$, $G_2=K_1$, $Y_1=Y$, and with $L(y_2)\in\Lc(Y,Y)$ being defined
as the multiplication operator $L(y_2)y(t)=G_2(y_2)(t)y(t)$. Since $Y$ is a
preideal$^*$ space, we have that~\eqref{e:ltG2} implies~\eqref{e:ltL}.
Moreover, for every $y_1\in U_1$, the function $K_0(y_1)$ belongs to the range
of $L(x)$, because a preimage is given by~\eqref{e:preimage}: Note that this
preimage belongs to $Y$, since $Y$ is preideal$^*$.
Hence, also~\eqref{e:ltrange} is satisfied.
\end{proof}

The following result implies in particular that, if the hypotheses of
Corollary~\ref{c:Liptrans} are satisfied with some locally regular
preideal space $X=X_1=X_2=Y$ at some $x=x_1=x_2$,
and if $G$ satisfies a local Lipschitz condition in the
\emph{second} variable in a neighborhood of $(x,x)$ (which is not restrictive,
since $G(y,\parameter)$ is an Urysohn operator), then a local Lipschitz
condition for the generalized fully nonlinear integral operator $F(y)=G(y,y)$
in a neighborhood of $x$ implies that $G(\parameter,x)$ satisfies
a \emph{global} Lipschitz condition, even pointwise.

\begin{theorem}[Lipschitz-Degeneration]\label{t:fullLip}
Let $X$, $Y$ be locally regular preideal spaces with the same real form
$X_{\R}=Y_{\R}$. Let $U\subseteq X$, and
$F\colon U\to Y$ be a $G$-abstract fully nonlinear integral operator at
some interior point $x_0$ of $U$, say $F(x)=G(x,x)$ with
$G(\parameter,x_0)\colon V\to Y$ being locally determined.

Let $G$ be locally $(\tau,\ell)$-Lipschitz transferrable at $(x_0,x_0)$.
In case $\tau>0$ suppose also
that $G$ satisfies a Lipschitz condition with constant $L_2$ in the
second variable in a neighborhood of $(x_0,x_0)$, and that $F$ satisfies
a local Lipschitz condition at $x_0$ with constant $L>0$.

Then $G(\parameter,x_0)$ satisfies the global pointwise Lipschitz condition
\begin{equation}\label{e:Lippoint}
\Abs{G(y_1,x_0)(t)-G(y_2,x_0)(t)}\le L_1\Abs{y_1(t)-y_2(t)}
\quad\text{for almost all $t\in\Omega$}
\end{equation}
with $L_1\coloneqq\ell+\tau(L+L_2)$ ($L_1\coloneqq\ell$ in case $\tau=0$)
for every $y_1,y_2\in V$.
\end{theorem}

Theorem~\ref{t:fullLip} holds for a more general class of spaces and
operators, namely whenever a corresponding degeneracy result for
$G_0=G(\parameter,x_0)$ is available. In fact, we show first that the
hypotheses imply that $G_0$ satisfies a \emph{local} Lipschitz condition
with constant $L_1$.

\begin{lemma}\label{l:localLip}
Let $X,Y$ be normed spaces, $M\subseteq X\times X$, and
$G\colon M\to Y$ be locally $(\tau,\ell)$-Lipschitz transferrable
at $(x_0,x_0)$. In case $\tau>0$ suppose also
that $G$ satisfies a Lipschitz condition with constant $L_2$ in the
second variable in a neighborhood of $(x_0,x_0)$, and that $F$ satisfies
a local Lipschitz condition at $x_0$ with constant $L>0$.
Then $G(\parameter,x_0)$ satisfies a Lipschitz condition with constant
at most $L_1\coloneqq\ell+\tau(L+L_2)$ ($L_1\coloneqq\ell$ in case $\tau=0$)
in some neighborhood $U\subseteq X$ of $x_0$.
\end{lemma}
\begin{proof}
Let $U$ be as in Definition~\ref{d:Lip}. Shrinking $U$ if necessary, we can
assume that $G(x,\parameter)$ satisfies a local Lipschitz condition on $U$
with constant $L_2$ for every $x\in U$. We claim that $U$ has the required
property. Thus, let $x,y\in U$, without loss of generality
$\Norm{y-x_0}\le\Norm{x-x_0}$.
Putting $h\coloneqq y-x_0$ in Definition~\ref{d:Lip}, we obtain
\[\Norm{G(x,x_0)-G(y,x_0)}\le\tau\Norm{G(x,y)-G(y,y)}+\ell\Norm{x-y}\text,\]
so that the claimed Lipschitz condition on $U$ follows in case $\tau=0$
immediately, and in case $\tau>0$ together with
\[\Norm{G(x,y)-G(y,y)}=\Norm{G(x,y)-G(x,x)+F(x)-F(y)}\le
(L_2+L)\Norm{x-y}\text.\]
\end{proof}

\begin{proof}[Proof of Theorem~\ref{t:fullLip}]
By~\eqref{l:localLip}, we know that $G_0\coloneqq G(x_0+\parameter,x_0)$
satisfies a Lipschitz condition with constant at most $L_1$ in some
neighborhood $U\subseteq X$ of $0$. We have to show that $G_0$ satisfies the
global pointwise Lipschitz condition~\eqref{e:Lippoint}. Assume by
contradiction that this is not the case, that is, there are $y_1,y_2\in V-x_0$,
$C>L_1$, and a set $T\subseteq\Omega$ of positive measure such that
\[\Abs{G_0(y_1)(t)-G_0(y_2)(t)}>\Abs{Cy_1(t)-Cy_2(t)}=C\Abs{y_1(t)-y_2(t)}\]
for almost all $t\in T$.
Since the left-hand side is strictly positive and $G_0$ is locally determined,
we can assume (shrinking $T$ if necessary) that $y_1(t)\ne y_2(t)$ for
all $t\in T$, hence $T\subseteq\supp(y_1-y_2)$.
Since $X$ is locally regular, we find for every $\varepsilon>0$ a set
$D\subseteq T$ of positive measure such that the function
$z(t)=\Abs{y_1(t)}+\Abs{y_2(t)}$ satisfies
$\Norm{P_Dz}_{X_{\R}}<\varepsilon$. Then also $x_i=P_Dy_i$ satisfy
$\Norm{P_Dy_i}<\varepsilon$, because $\Abs{x_i(t)}\le P_Dz(t)$ for $i=1,2$.
In particular, choosing $\varepsilon>0$ sufficiently small,
we can assume that $x_1,x_2\in U$. Since $G_0$ is locally determined, we have
\begin{gather*}
\Abs{G_0(x_1)(t)-G_0(x_2)(t)}\ge\Abs{P_D(G_0(x_1)-G_0(x_2))(t)}\\
\ge\Abs{P_D(Cy_1-Cy_2)(t)}=C\Abs{x_1(t)-x_2(t)}
\end{gather*}
for almost all $t\in\Omega$. Since $X$ and $Y$ are preideal spaces with the
same real form $X_{\R}=Y_{\R}$, $D\subseteq\supp(y_1-y_2)$, and $\mes D>0$,
we conclude
\[\Norm{G_0(x_1)-G_0(x_2)}\ge C\Norm{x_1-x_2}\ne0\text.\]
Hence, $G_0$ fails to satisfy a local Lipschitz condition on $U$ with
constant $L_1$, contradicting the beginning of the proof.
\end{proof}

Theorem~\ref{t:fullLip} contains the folklore result that for superposition
operators in $L_p$-spaces local and global Lipschitz conditions are equivalent.
In contrast to the results of this type which we found in literature
(e.g.~\cite[Theorem~3.10]{Appell}), we do not require that $\Omega$ has
finite (or $\sigma$-finite) measure, and moreover, we treat the space of
vector functions:

\begin{corollary}[Special Case of Superposition Operators]\label{c:Lip}
Let $X$ and $Y$ be locally regular preideal spaces with the same real form,
e.g.\ $X=L_p(\Omega,E_1)$ and $Y=L_p(\Omega,E_2)$ with $1\le p<\infty$
and $\Omega$ containing no atoms of finite measure.
If $F(x)(t)=f(t,x(t))$ is such that $F$ acts from $U\subseteq X$ into $Y$
and satisfies a local Lipschitz condition near some interior point $x_0$ of
$U$ with constant $L$, then $F$ satisfies the global pointwise
Lipschitz condition
\[\Abs{F(y_1)(t)-F(y_2)(t)}\le L\Abs{y_1(t)-y_2(t)}\quad\text{for almost all
$t\in\Omega$}\]
for every $y_1,y_2\in U$ with the same constant $L$.
\end{corollary}
\begin{proof}
Put $G(x_1,x_2)=F(x_1)$, and note that $G$ is locally
$(0,L)$-Lipschitz transferrable by Proposition~\ref{p:Liptrans}.
Thus, the result follows from Theorem~\ref{t:fullLip}.
\end{proof}

\section{Local Darbo and Local Compactness Condition}

Let $X$ be a normed space. For $r\in[0,\infty]$, we put
$B_r(x_0)=\Curly{x\in X:\Norm{x-x_0}<r}$. For $M\subseteq X$, the
\emph{Hausdorff measure of noncompactness} $\alpha(M)$ of a set
$M\subseteq X$ is defined as the infimum of all $\varepsilon\in[0,\infty]$
such that $M$ has a finite $\varepsilon$-net $N\subseteq X$, that is,
$M\subseteq\bigcup_{x\in N}B_\varepsilon(x)$. Recall that as a consequence of
Riesz's lemma one easily obtains
\[\alpha(B_r(x_0))=\begin{cases}0&\text{if $\dim X<\infty$ and $r<\infty$}\\
r&\text{otherwise.}\end{cases}\]
If $M\subseteq X$ and $Y$ is a normed space, then a map
$F\colon M\to Y$ satisfies the \emph{Darbo} condition
with a constant $\ell\in[0,\infty)$ if
\begin{equation}\label{e:ball}
\alpha(F(K))\le\ell\alpha(K)\quad\text{for all $K\subseteq M$.}
\end{equation}
A Darbo condition with constant $\ell=0$ means that $F(M_0)$ is precompact.
However, many results for compact maps hold also if $\ell$ is only
sufficiently small. For instance, a variant of the
famous fixed point theorem of Darbo states
that if $M_0$ is a nonempty complete bounded convex subset of
$X=Y$ and $F\colon M_0\to M_0$ is continuous and satisfies a
Darbo condition with constant $\ell<1$ then $F$ has a fixed point.

There are many examples of maps which satisfy a Darbo condition with
constant $\ell$. For instance, any Lipschitz map $F\colon X\to Y$ with
constant $\ell$ satisfies a Darbo condition with constant $\ell$. Moreover,
any compact perturbation of such a map has the same property, too.
The latter shows, in particular, that the condition~\eqref{e:ball} does
in general not imply a linear growth condition on $\diam(F(B_r(x_0)))$
with respect to $r$.

We will show in this section that this is different in case of
($G$-abstract) fully nonlinear integral operators.

Note that the special case of~\eqref{e:ball} when one considers only sets $K$
of the form $K=B_r(x_0)$ simplifies (in the interesting
case that $X$ is infinite-dimensional) to
\[\frac{\alpha(F(B_r(x_0)))}r\le\ell\text.\]
Since we are interested in a local Darbo condition, this motivates the
notation
\begin{equation}\label{e:Fest}
[F]_{x_0}=\limsup_{r\to0}\,\frac{\alpha(F(B_r(x_0)))}r\text.
\end{equation}
If $X$ is infinite-dimensional then~\eqref{e:Fest} is
(up to an arbitrary small error) the smallest constant $\ell$
for which one has~\eqref{e:ball} at least for sets of the form
$K=B_r(x_0)$ with small $r>0$.

\begin{proposition}
If $L\colon X\to Y$ is linear then $[L]_{x_0}=\alpha(L(B_1(0)))$ for each
$x_0\in X$.
\end{proposition}
\begin{proof}
Trivially, $\alpha(M+y_0)=\alpha(M)$  and $\alpha(rM)=r\alpha(M)$ for every
$M\subseteq Y$, $y_0\in Y$ and $r\in(0,\infty)$. Applying this with
$M=L(B_1(0))$ and $y_0=L(x_0)$, we obtain
$\alpha(L(B_r(x_0)))=r\alpha(LB_1(0))$ which implies the assertion.
\end{proof}

We need one further definition. Let $c\ge1$.
A preideal$^*$ space $Y$ is called \emph{$(\alpha,c)$-nondegenerate}
if for each $y_1,y_2\in Y$ there holds
\begin{equation}\label{e:nondeg}
\alpha(\Curly{P_Dy_1+P_{\Omega\setminus D}y_2:
\text{$D\subseteq\Omega$ measurable}})\ge\frac1{2c}\Norm{y_1-y_2}\text.
\end{equation}
Note that $c=1$ is the smallest constant which can occur
in~\eqref{e:nondeg} (if $y_1\ne y_2$). Indeed, the converse estimate
\[\alpha(\Curly{P_Dy_1+P_{\Omega\setminus D}y_2:
\text{$D\subseteq\Omega$ measurable}})\le\frac12\Norm{y_1-y_2}\]
holds unconditionally, since $z=\frac12(y_1-y_2)$ constitutes in view of
\[\Abs{P_Dy_1(t)+P_{\Omega\setminus D}y_2(t)-z(t)}\le\Abs{z(t)}\]
an $\frac12\Norm{y_1-y_2}$-net for the set on the left-hand side.
In particular, for $(\alpha,1)$-nondegenerate spaces,
there holds equality in~\eqref{e:nondeg}. These spaces are simply called
$\alpha$-nondegenerate in~\cite{Appell}.

While for real-valued functions some convenient sufficient criteria for
$\alpha$-nondegenerate ideal spaces have been given in~\cite{Appell},
it is not immediately clear that these hold also for spaces of
vector functions. In the appendix, we will provide a sufficient conditions
which even for scalar functions extends that from~\cite{Appell}.
As a special case of that result we obtain (cf.\ Theorem~\ref{t:Lphinondeg}):

\begin{proposition}
Suppose that $\Omega$ contains no atoms of finite measure. Let
$1\le p\le\infty$; in case $p=\infty$ suppose in addition that $\Omega$
is $\sigma$-finite. Then $L_p(\Omega,E_2)$ is $(\alpha,1)$-nondegenerate.
\end{proposition}

Now we are in a position to formulate the
main result concerning the Darbo condition:

\begin{theorem}[Darbo-Degeneration]\label{t:Darbo}
Let $X$ and $Y$ be spaces of measurable functions over $\Omega$, $Y$ being
an $(\alpha,c)$-nondegenerate preideal$^*$ space. Let $U\subseteq X$, and
$F\colon U\to Y$ be a $G$-abstract fully nonlinear integral operator at
$x_0\in U$, say $F(x)=G(x,x)$ with $G(\parameter,x_0)\colon V\to Y$ being
locally determined. Suppose in addition that $G$ is
diagonal-differentiable at $(x_0,x_0)$ with respect to the second variable
with diagonal derivative $D_2G(x_0,x_0)$. Then
\[\limsup_{r\to0}\,\frac{\diam(G(B_r(x_0),x_0))/2}r\le
c([F]_{x_0}+[D_2G(x_0,x_0)]_0)\text.\]
If additionally $X$ and $Y$ are locally regular preideal spaces
with the same real form $X_{\R}=Y_{\R}$ then $G(\parameter,x_0)$
satisfies the global pointwise diameter-growth condition
\begin{equation}\label{e:growth}
\begin{gathered}
\Abs{G(x_1,x_0)(t)-G(x_2,x_0)(t)}\le([F]_{x_0}+[D_2G(x_0,x_0)]_0)\cdot\\
2c\max\Curly{\Abs{x_1(t)-x_0(t)},\Abs{x_2(t)-x_0(t)}}\quad
\text{for almost all $t\in\Omega$}
\end{gathered}
\end{equation}
for every $x_1,x_2\in V$.
\end{theorem}

\begin{corollary}[Compact-Degeneration]\label{c:compact}
Assume that the hypotheses of the first part of
Theorem~\ref{t:Darbo} are satisfied.
If $F|_{B_r(x_0)}$ is compact for some $r>0$ and if $D_2G(x_0,x_0)$ is compact
then $G(\parameter,x_0)$ is Fr\'{e}chet differentiable at $x_0$ with
derivative $0$. If additionally $(X,Y)$ is a weak $V$-pair then
$G(\parameter,x_0)$ is constant.
\end{corollary}
\begin{proof}[Proof of Corollary~\ref{c:compact}]
Theorem~\ref{t:Darbo} implies $\diam(G(B_r(x_0),x_0))/r\to0$ as $r\to0$.
Hence, $\Norm{G(x_0+h,x_0)-G(x_0,x_0)}/\Norm h\to0$ as $\Norm h\to0$, and
so $G_0=G(\parameter,x_0)$ is Fr\'{e}chet differentiable at $x_0$ with
derivative $0$. If $(X,Y)$ is a weak $V$-pair, we can apply
Theorem~\ref{t:fulldiff} with $\widetilde G(x,y)=G_0(x)$ in view
of Proposition~\ref{p:diagonal}\iref{i:independent} to obtain that
$G_0$ is bounded affine. Since its derivative vanishes, $G_0$ must be constant.
\end{proof}

Actually, Theorem~\ref{t:Darbo} holds also for any other class of spaces
and operators for which degeneration results for the corresponding operator
$G_0=G(\parameter,x_0)$ are available. In fact, the crucial ingredient for the
proof is the following general result:

\begin{lemma}\label{l:Darbo}
Let $X$ and $Y$ be normed spaces, $M\subseteq X\times X$, and let
$G\colon M\to Y$, and let $F(x)\coloneqq G(x,x)$.
If $G$ is diagonal-differentiable with respect to the second variable at some
$(x_0,x_0)$ with diagonal derivative $L=D_2G(x_0,x_0)$, then
\[\relax[G(\parameter,x_0)]_{x_0}\le[F]_{x_0}+\alpha(L(B_1(0)))\text.\]
\end{lemma}
\begin{proof}
Let $L_F>[F]_{x_0}$, $L_G>[L]_0=\alpha(L(B_1(0)))$ and $\varepsilon>0$.
Then for all sufficiently small $r>0$ the following holds: There is a finite
$rL_F$-net $N_F\subseteq Y$ for $F(B_r(x_0))$, there is a finite
$rL_G$-net $N_G\subseteq Y$ for $L(B_r(0))$, and
\[\Norm{G(x_0+h,x_0+h)-G(x_0+h,x_0)-Lh}\le\varepsilon\Norm h
\quad\text{if $\Norm h\le r$.}\]
For $x\in B_r(x_0)$ put $h\coloneqq x-x_0$. There are
$x_F\in N_F$ and $x_G\in N_G$ with $\Norm{F(x)-x_F}\le rL_F$ and
$\Norm{Lh-x_G}\le rL_G$, and so
\begin{gather*}
\Norm{x_F-x_G-G(x,x_0)}=\\
\Norm{(x_F-F(x))+(G(x,x)-G(x,x_0)-Lh)+(Lh-x_G)}
<rL_F+\varepsilon r+rL_G\text.
\end{gather*}
Hence, $N_F-N_G$ constitutes a finite $r(L_F+L_G+\varepsilon)$-net for
$G(B_r(x_0),x_0)$, and so $[G(\parameter,x_0)]_{x_0}\le L_F+L_G+\varepsilon$.
\end{proof}

\begin{proof}[Proof of the first part of Theorem~\ref{t:Darbo}]
Let $C>[F]_{x_0}+[D_2G(x_0,x_0)]_0$. By Lemma~\ref{l:Darbo}, we have
$C>[G_0]_{x_0}$ with $G_0=G(\parameter,x_0)$. Hence, there is $\delta>0$ such
that for each $r\in(0,\delta)$ the estimate $\alpha(G_0(B_r(x_0)))\le Cr$
holds. Now if $x_1,x_2\in B_r(x_0)$ are arbitrary and $D\subseteq\Omega$ is
measurable then also $y=P_Dx_1+P_{\Omega\setminus D}x_2\in B_r(x_0)$ and,
since $G_0$ is locally determined,
\[G_0(y)=P_DG_0(x_1)+P_{\Omega\setminus D}G_0(x_2)\text.\]
Consequently,
\[G_0(B_r(x_0))\supseteq M_{x_1,x_2}=
\Curly{P_DG_0(x_1)+P_{\Omega\setminus D}G_0(x_2):
\text{$D\subseteq\Omega$ measurable}}\text.\]
Since $Y$ is $(\alpha,c)$-nondegenerate, we obtain
\[Cr\ge\alpha(G_0(B_r(x_0)))\ge\alpha(M_{x_1,x_2})\ge
\frac1{2c}\Norm{G_0(x_1)-G_0(x_2)}\text.\]
Since $x_1,x_2\in B_r(x_0)$ are arbitrary, we have shown that
$\diam(G_0(B_r(x_0)))\le2cCr$ for all $r\in(0,\delta)$, and the
assertion follows.
\end{proof}

The first part of Theorem~\ref{t:Darbo} implies the second part in view of
the following global growth result for locally determined operators:

\begin{theorem}[Global Growth Degeneration for Superposition Operators]
Let $X$ and $Y$ be locally regular preideal spaces with the same real form
$X_{\R}=Y_{\R}$. Let $x_0$ be an interior point of $U\subseteq X$,
and $F\colon U\to Y$ be locally determined. If
\[B=\limsup_{r\to0}\,\frac{\diam(F(B_r(x_0)))}r\]
is finite then $F$ satisfies the global pointwise diameter-growth condition
\[\Abs{F(x)(t)-F(y)(t)}\le B\max\Curly{\Abs{x(t)-x_0(t)},\Abs{y(t)-y_0(t)}}
\quad\text{for almost all $t\in\Omega$}\]
for every $x,y\in U$.
\end{theorem}
\begin{proof}
Without loss of generality, we can assume that $x_0=0$, since the general
case follows by applying this special case with the locally determined
map $F(\parameter-x_0)$ on the domain $U-x_0$.
Hence, if the conclusion would be false, we could find $x,y\in U$,
$C>B$, and a set $T\subseteq\Omega$ of positive measure such that
\[\Abs{F(x)(t)-F(y)(t)}>C\Abs{x(t)}\ge C\Abs{y(t)}\]
holds for almost all $t\in T$. Since $F$ is locally determined and the
left-hand side is positive, we can assume (shrinking $T$ if necessary)
that $x(t)\ne y(t)$ for all $t\in T$, since $\Abs{x(t)}\ge\Abs{y(t)}$
in particular $x(t)\ne0$ for almost all $t\in T$, hence $T\subseteq\supp x$.

By hypothesis, there is $r_0>0$ such that
$\diam(F(B_r(0)))<Cr$ for all $r\in(0,r_0)$. Since $X$ is locally regular,
we find a set $D\subseteq T$ of positive measure such that $x_1\coloneqq P_Dx$
satisfies $r\coloneqq\Norm{x_1}<r_0$. Note that indeed $r>0$, because
$D\subseteq\supp x$ has positive meausre. Moreover, since $x_2\coloneqq P_Dy$
satisfies $\Abs{x_2(t)}\le\Abs{x_1(t)}$, we have also $\Norm{x_2}<r$.
Since $F$ is locally determined, we find
\[\Abs{F(x_1)(t)-F(x_2)(t)}\ge\Abs{P_D(F(x_1)(t)-F(x_2)(t))}\ge
C\Abs{x_1(t)}\text.\]
Since $X$ and $Y$ have the same real form, we obtain
\[\diam(F(B_r(0)))\ge\Norm{F(x_1)-F(x_2)}\ge C\Norm{x_1}=Cr\text,\]
which is a contradiction.
\end{proof}

The hypotheses of Corollary~\ref{c:compact} can easily be verified
by the methods already used in Section~\ref{s:diff}. In particular,
we obtain for the operator arising in continuous limits of coupled
Kuramoto systems:

\begin{example}\label{x:compact}
Let $\Omega$ be $\sigma$-finite. For $1<p\le q<\infty$, let
\[F(x)(t)=\int_\Omega k(t,s)\sin\bigl(x(t)-x(s)\bigr)\ds\text,\]
where $k\colon\Omega\times\Omega\to\R$ is measurable and satisfies the estimate
$\Abs{k(t,s)}\le a(t)b(s)$ with some $a\in L_q(\Omega)$,
$b\in L_{p/(p-1)}(\Omega)$.
If $k\ne0$ then the restriction of $F\colon L_p(\Omega)\to L_q(\Omega)$ to
every nonempty open set fails to be compact.

Indeed, we have $F(x)=G(x,x)$ with
\[G(x_1,x_2)(t)=\int_\Omega k(t,s)\sin\bigl(x_1(t)-x_2(s)\bigr)\ds\text.\]
On the one hand, Theorem~\ref{t:fulldiffx} and
Proposition~\ref{p:diagonal}\iref{i:contderiv} imply that
$G$ is diagonal-differentiable at every $(x_1,x_2)$ with diagonal derivative
\[D_2G(x_1,x_2)h(t)=-\int_\Omega k(t,s)\cos\bigl(x_1(t)-x_2(s)\bigr)h(s)\ds
\text,\]
and since $D_2G(x_1,x_2)$ is a compact operator from $X\coloneqq L_p(\Omega)$
into $Y\coloneqq L_q(\Omega)$ for every $(x_1,x_2)\in X\times X$
(see e.g.~\cite[Corollary~9.19]{VaethVolt}), we obtain from
Corollary~\ref{c:compact}: If $F|_{B_r(x_0)}$ is compact for some $r>0$
and some $x_0\in X$ then $G(\parameter,x_0)$ is constant.
On the other hand, if $k\ne0$ then $G(\parameter,x_0)$ is non-constant
for any $x_0\in L_p(\Omega)$.
\end{example}

In case $p=q$, Example~\ref{x:compact} can actually be made quantitative for
particular $k$. Indeed, since $\sin$ growth near $0$ almost like the identity,
one can (depending on $k$) obtain for every $x_0\in X$ some $C>0$ and a
function $h\in X$ such that
\[\Abs{G(x_0+h,x_0(t))(t)-G(x_0,x_0)(t)}>C\Abs{h(t)}\]
for all $t$ on a set $T\subseteq\Omega$ of positive measure.
Then for every such $C$, one has the lower Darbo type estimate
$[F]_{x_0}>C/2$ by Theorem~\ref{t:Darbo}, in particular
\[\alpha(F(B_r(x_0)))\ge\frac C2\alpha(B_r(x_0))\]
for all sufficiently small $r>0$.

\appendix

\section{Mean Value Theorem for Vector Functions}

The mean value theorem for vector functions is of course well-known
and usually obtained by Hahn-Banach. We point out that the following proof
does \emph{not} require any form of the axiom of choice.

\begin{lemma}\label{l:mean}
Let $X$ be a normed space, and $\varphi\colon[a,b]\to E$ be continuous.
If $\varphi$ is differentiable in each point of $(a,b)$
and the derivative is bounded by $M$ then
$\Norm{\varphi(b)-\varphi(a)}\le M(b-a)$.
\end{lemma}
\begin{proof}
Given $N>M$ and $a_0\in(a,b)$, we are to show that
\[I\coloneqq\Curly{t\in[a_0,b]:\Norm{\varphi(t)-\varphi(a_0)}\le N(t-a_0)}\]
contains $b$. Note that $b_0\coloneqq\sup I$ exists (because $a_0\in I$) and
belongs to $I$. If $b_0<b$ then $\Norm{\varphi'(b_0)}<N$ and $b_0\in I$ imply
\[\Norm{\varphi(t)-\varphi(a_0)}\le
\Norm{\varphi(t)-\varphi(b_0)}+\Norm{\varphi(b_0)-\varphi(a_0)}\le
N(t-b_0)+N(b_0-a)=N(t-a)\]
if $t\in(b_0,b]$ is sufficiently close to $b_0$, that is $b_0<t\in I$,
contradicting $b_0=\max I$.
\end{proof}

\section{$(\alpha,c)$-nondegenerate Preideal Spaces}

Throughout this section, let $(E,\Abs\parameter)$ be a Banach space,
and $X$ be a preideal space of measurable functions $x\colon\Omega\to E$.

The arguments in this section follow essentially~\cite[Lemma~2.9]{Appell},
but we cover also the case $E\ne\R$ and $\mes\Omega=\infty$ and relax the
main hypothesis so that in contrast to~\cite{Appell}, we can treat
Orlicz-Musielak spaces (that is, with $t$-dependent Young functions)
with the same result.

The following result is perhaps not so obvious in case $E\ne\R$:

\begin{lemma}\label{l:partition}
Suppose that each $x\in X$ has $\sigma$-finite support.
Then the set of all functions $y\in X$ of the form
\[y(t)=\sum_{n=1}^\infty u_n\charac{D_n}(t)\]
with $u_n\in E$ and pairwise disjoint measurable $D_n\subseteq E$
is dense in $X$.
\end{lemma}
\begin{proof}
Let $x\in X$ and $\varepsilon>0$ be fixed. Since $x$ is measurable and
$\supp x$ is $\sigma$-finite, we can assume,
modifying $x$ on some null set if necessary, that
$x(\Omega)\subseteq E$ is separable, see
e.g.~\cite[Corollary~1.1]{VaethIntegral}. Let $\Curly{u_1,u_2,\dotsc}$
be dense in $x(\Omega)$. Let $I_k$ be a partition
of $(0,\infty)$ into disjoint intervals, $e_k=\inf I_k$, and
$T_k=\Curly{t\in\Omega:\Abs{x(t)}\in I_k}$. We define the Borel sets
\[B_{k,n}\coloneqq B_{\varepsilon e_k}(u_n)\setminus
\bigcup_{m<n}B_{\varepsilon e_k}(u_m)\text.\]
For each fixed $k$, the family $\Curly{B_{k,n}:n}$ is pairwise disjoint, and
its union contains $x(\Omega)$. Hence, the function
\[y(t)=\sum_{k,n=1}^\infty u_n\charac{T_k\cap x^{-1}(B_{k,n})}(t)\]
has the required form. By construction, we have
\[\Abs{x(t)-y(t)}\le\sum_{k=1}^\infty\varepsilon e_k\charac{T_k}(t)\le
\varepsilon\Abs{x(t)}\]
for all $t\in\Omega$, and so $x-y\in X$ (hence $y\in X$), and
$\Norm{x-y}\le\varepsilon\Norm x$.
\end{proof}

Let $c\ge1$. In a generalization of~\cite[Section~2.6]{Appell},
we say that $X$ is \emph{$c$-average-stable} if for each sequence
$D_n\subseteq\Omega$ of pairwise disjoint measurable sets and
each sequences of numbers $a_{j,n}\ge0$ $(j=1,2)$ for which
\begin{equation}\label{e:xaverage}
x(t)=\sum_{n=1}^\infty\frac12(a_{1,n}+a_{2,n})\charac{D_n}(t)
\end{equation}
belongs to the real form $X_{\R}$ of $X$, and for each $\varepsilon>0$,
there is a refinement of the partition $D_n$ (which we denote again by $D_n$)
such that whenever $D_n$ divides into two sets $D_{1,n},D_{2,n}$ of
equal measure, the function
\begin{equation}\label{e:waverage}
w(t)=\sum_{n=1}^\infty\bigl(a_{1,n}\charac{D_{1,n}}(t)+
a_{2,n}\charac{D_{2,n}}(t)\bigr)
\end{equation}
satisfies the estimate $\Norm x_{X_{\R}}\le c\Norm w_{X_{\R}}+\varepsilon$.

Note that we have automatically $w\in X_{\R}$, because $w(t)\le2x(t)$.
If $X_{\R}$ is average-stable in the sense of~\cite[Section~2.6]{Appell},
then $X$ is $1$-average-stable, because, in the notation
of~\cite[Section~2.6]{Appell}, we have $x=P_\omega w$ if $\mes\Omega<\infty$
(as assumed in~\cite{Appell}), and the norm of the second associate space
of $X_{\R}$ coincides with that of $X_{\R}$ if $X$ is almost perfect,
see e.g.~\cite[Corollary~3.4.4]{VaethIdeal}.

Our definition is more technical than that from~\cite{Appell},
but it has the advantage that we can verify that Orlicz-Musielak spaces
are $1$-average stable, as we will see in Theorem~\ref{t:Lphinondeg}.
In contrast, it is unclear whether Orlicz-Musielak spaces are
average-stable in the sense of~\cite{Appell}.

\begin{theorem}\label{t:nondeg}
Suppose that for each $x\in X$ the set $\supp x$ is $\sigma$-finite and
contains no atoms of finite measure. If $X$ is $c$-average-stable then $X$ is
$(\alpha,c)$-nondegenerate.
\end{theorem}
\begin{proof}
For $y_1,y_2\in X$, let $R(y_1,y_2)$ denote the set occuring
in~\eqref{e:nondeg}. Let $\alpha_0>\alpha(R(y_1,y_2))$ and $\varepsilon>0$
be arbitrary. There is a finite $\alpha_0$-net for $R(y_1,y_2)$, consisting
of functions $z_1,\dotsc,z_m\in X$.
In view of Lemma~\ref{l:partition} and by considering a common refinement
(see e.g.~\cite[Lemma~1.1]{VaethIdeal}), we can assume without loss of
generality that there are disjoint measurable sets $D_n\subseteq\Omega$
such that
\[y_j=\sum_{n=1}^\infty u_{j,n}\charac{D_n}\text,\quad
z_k=\sum_{n=1}^\infty v_{k,n}\charac{D_n}\]
for $j=1,2$ and $k=1,\dotsc,m$. Since we can assume that $\bigcup D_n$ is
$\sigma$-finite, we can also assume, refining the partition if necessary,
that each $D_n$ has finite measure. We consider now the family of functions
\[x_k(t)=\frac12\bigl(\Abs{y_1(t)-z_k(t)}+\Abs{y_2-z_k(t)}\bigr)\]
from $X_{\R}$. Note that
\[x_k(t)=\sum_{n=1}^\infty\frac12(a_{1,k,n}+a_{2,k,n})\charac{D_n}(t)\quad
(k=1,\dotsc,m)\]
with $a_{j,k,n}=\Abs{u_{j,n}-v_{k,n}}$. After possibly passing to a further
common refinement of the family $D_n$, the hypothesis thus implies that
we can divide $D_n$ into two disjoint measurable sets $D_{j,n}$ $(j=1,2)$
of equal measure, and the function
\[w_k(t)=\sum_{n=1}^\infty\bigl(a_{1,k,n}\charac{D_{1,n}}(t)+
a_{2,k,n}\charac{D_{2,n}}(t)\bigr)\]
satisfies $\Norm{x_k}\le c\Norm{w_k}+\varepsilon$. We put now
\[y(t)=\sum_{n=1}^\infty\bigl(u_{1,n}\charac{D_{1,n}}(t)+
u_{2,n}\charac{D_{2,n}}(t)\bigr)\text.\]
Then $y\in R(y_1,y_2)$, and so there is some $k$ with
$\Norm{y-z_k}\le\alpha_0$. Since $\Abs{w_k(t)}=\Abs{y(t)-z_k(t)}$,
we conclude for this $k$ that $\Norm{w_k}\le\alpha_0$, and we thus obtain
in view of
\[\Abs{y_1(t)-y_2(t)}\le\Abs{y_1(t)-z_k(t)}+\Abs{z_k(t)-y_2(t)}=2x_k(t)\]
that
\[\Norm{y_1-y_2}\le2\Norm{x_k}\le2(c\Norm{w_k}+\varepsilon)\le
2c\alpha_0+2\varepsilon\text.\]
Hence,~\eqref{e:nondeg} follows.
\end{proof}

Let $\Phi\colon\Omega\times[0,\infty)\to[0,\infty]$ be a generalized
Young-function, that is, for almost all $t\in\Omega$ the function
$\Phi(t,\parameter)$ is convex, not identically $0$,
and $\Phi(t,0)=0$, and $\Phi(\parameter,u)$ is measurable for every $u>0$.
Then $\Phi$ induces the Orlicz-Musielak space $X=L_\Phi(\Omega,E)$
which consists of all measurable functions $x\colon\Omega\to E$ such
that there is $\lambda>0$ such that $x/\lambda$ belongs to the set
\[M_\Phi=\Curly{x:\int_\Omega\Phi(t,\Abs{x(t)})\dt\le1}\text.\]
The infimum of all those $\lambda$ is the \emph{Luxemburg} norm $\Norm x$.
Since $M_\Phi$ is convex with $0\in M_\Phi=-M_\Phi$ and the Luxemburg norm
is the corresponding Minkowski functional, it follows that $X$
is indeed normed (and a preideal space) by the Luxemburg norm.
Note that the spaces $L_{p(\parameter)}(\Omega,E)$ with measurable
$p\colon\Omega\to[1,\infty]$ are special cases of Orlicz-Musielak spaces with
\[\Phi(t,u)=\begin{cases}u^{p(t)}&\text{if $p(t)<\infty$,}\\
0&\text{if $p=\infty$, $u\le1$,}\\
\infty&\text{if $p=\infty$, $u>1$.}\end{cases}\]
In case of constant $p$, these spaces reduce to the classical Bochner-Lebesgue
spaces $L_p(\Omega,E)$ with the usual norm.

\begin{theorem}\label{t:Lphinondeg}
Suppose that $\Curly{t\in\Omega:0\in\Phi(t,(0,\infty))}$ is
$\sigma$-finite, and that $\Omega$ contains no atoms of finite measure.
Then the Orlicz-Musielak space $X=L_\Phi(\Omega,E)$ with the Luxemburg norm
is $1$-average-stable and $(\alpha,1)$-nondegenerate.
\end{theorem}
\begin{proof}
Since for every $x\in X$ some positive multiple of $x$ belongs to $M_\Phi$,
the first hypothesis implies that $\supp x$ is $\sigma$-finite.
By Theorem~\ref{t:nondeg}, it thus suffices to show that
$X$ is $1$-average-stable. Consider $x\in X$ as in~\eqref{e:xaverage}.
Since $\supp x$ is $\sigma$-finite, we can assume by refining the
partition $D_n$ if necessary, that each $D_n$ has finite measure.

We are to show that for every $\lambda,\varepsilon>0$ with
$\lambda<\Norm x_{X_{\R}}/(1+\varepsilon)$ there is a
refinement of the partition $D_n$ such that every corresponding function $w$
from~\eqref{e:waverage} satisfies $\Norm w_{X_{\R}}>\lambda$.

To this end, we divide $[0,\infty]$ into the countably many disjoint sets
$\Curly0$, $\Curly\infty$, and
$[(1+\varepsilon)^k,(1+\varepsilon)^{k+1})$ ($k$ integer).
Refining the partition $D_n$ correspondingly, we thus can assume that there
are $m_{j,n}\in[0,\infty]$ (boundary points of the above intervals) with
\begin{equation}\label{e:phiest}
m_{j,n}\le\Phi\Bigl(t,\frac{a_{j,n}}\lambda\Bigr)\le
m_{j,n}(1+\varepsilon)\quad\text{for all $t\in D_n$ $(j=1,2)$.}
\end{equation}
Since $\Phi(t,\parameter)$ is convex, we thus find for $t\in D_n$ that
\[\Phi\Bigl(t,\frac{a_{1,n}+a_{2,n}}{2\lambda}\Bigr)\le
\frac12\Bigl(\Phi\Bigl(t,\frac{a_{1,n}}\lambda\Bigr)+
\Phi\Bigl(t,\frac{a_{1,n}}\lambda\Bigr)\Bigr)\le
\frac12(m_{1,n}+m_{2,n})(1+\varepsilon)\text,\]
which implies by the convexity of $\Phi(t,\parameter)$ and $\Phi(t,0)=0$ that
\begin{equation}\label{e:xestlow}
\Phi\Bigl(t,\frac{a_{1,n}+a_{2,n}}{2(1+\varepsilon)\lambda}\Bigr)\le
\frac1{1+\varepsilon}\Phi\Bigl(t,\frac{a_{1,n}+a_{2,n}}{2\lambda}\Bigr)\le
\frac12(m_{1,n}+m_{2,n})\text.
\end{equation}
On the other hand, since $\mes D_{j,n}=\frac12\mes D_n$ for $j=1,2$, we
obtain from~\eqref{e:phiest} that
\[\int_{D_n}\Phi\Bigl(t,\frac{w(t)}\lambda\Bigr)\dt\ge
\int_{D_n}\frac12(m_{1,n}+m_{2,n})\dt\text.\]
Using~\eqref{e:xestlow} and summing up over all $n$, we conclude
\[\int_\Omega\Phi\Bigl(t,\frac{w(t)}\lambda\Bigr)\dt\ge
\int_\Omega\Phi\Bigl(t,\frac{x(t)}{(1+\varepsilon)\lambda}\Bigr)\dt>1\text.\]
The last inequality holds, because $(1+\varepsilon)\lambda<\Norm x_{X_{\R}}$.
Hence, $\lambda<\Norm w_{X_{\R}}$, as claimed.
\end{proof}

We point out, that the novelty of our approach for Theorem~\ref{t:Lphinondeg}
is the case $E\ne\R$, and moreover, we can prove that $X=L_\Phi(\Omega)$ is
$1$-avarage-stable while it is unclear whether one can prove that
$X$ is average-stable in the sense of~\cite{Appell} (unless $\Phi$ is
independent of $t$).

\end{document}